\documentclass[12pt, parskip=full]{scrartcl}

\usepackage[ngerman, english]{babel}

\usepackage[utf8]{inputenc}

\setkomafont{sectioning}{\normalfont\normalcolor\bfseries}

\usepackage{amsmath}
\usepackage{amsfonts}
\usepackage{amssymb}
\usepackage{amsthm}
\usepackage{stmaryrd}

\theoremstyle{plain}
\newtheorem{theorem}{Theorem}[section]
\newtheorem{lemma}[theorem]{Lemma}
\newtheorem{corollary}[theorem]{Corollary}
\newtheorem{assumption}{Assumption}

\theoremstyle{definition}
\newtheorem{definition}{Definition}[section]

\theoremstyle{remark}
\newtheorem{remark}[theorem]{Remark}

\usepackage[round]{natbib}

\usepackage{url}

\usepackage{abstract}

\usepackage{hyperref}
\usepackage{graphicx}
\usepackage{xcolor}

\usepackage{pgf}
\usepackage{tikz}
\usetikzlibrary{arrows,automata,positioning,calc}

\usepackage{changes}

\title{Essential Stationary Equilibria of Mean Field Games with Finite State and Action Space}
\author{Berenice Anne Neumann\thanks{University of Trier, Department IV, Universitätsring 19, 54296 Trier, Germany}}

\begin{document}
	
	\maketitle
	
	\begin{abstract}
 		Mean field games allow to describe tractable models of dynamic games with a continuum of players, explicit interaction and heterogeneous states. Thus, these models are of great interest for socio-economic applications. A particular class of these models are games with finite state and action space, for which recently in \citet{NeumannComputation} a semi-explicit representation of all stationary equilibria has been obtained. In this paper we investigate whether these stationary equilibria are stable against model perturbations.
 		We prove that the set of all games with only essential equilibria is residual and obtain two characterization results for essential stationary equilibria.
	\end{abstract}
	
	\textbf{Keywords:} mean field game, essential equilibrium, stationary equilibrium, finite state space, finite action space
	
	\textbf{JEL Classifications:} C73, C72, C62 
	
	\section{Introduction}
	
	Mean field games have been introduced independently by \citet{LasryJapanese2007} and \citet{HuangNCE2006} as a game theoretic model for stochastic games in continuous time with a continuum of players.
	The main feature of these games is that the players do not observe the other players' behaviour individually, but only its distribution.
	These games allow for tractable models of the interaction of a continuum of players with explicit interaction (in contrast to the classical assumption in general equilibrium theory that ``prices mediate all social interaction'') as well as heterogeneous states (in contrast to representative agent models).
	
	This led to a large variety of economic applications (see \citet{GomesEconomicModelsMFG2015,ParisPrinceton2010, CainesHandbook}). 
	In particular, applications with finite state and action space where the dynamics of the individual player are given by a continuous time Markov chain have been considered.
	These include \citet{KolokoltsovBotnet2016}, \citet{KolokoltsovCorruption2017}, \citet{GueantDiss2009} as well as \citet{BesancenotParadigm2015} and the focus in all these applications lied in the analysis of stationary equilibria.
	Recently, also a formal model with finite state and action space has been introduced in \citet{DoncelPaper}, where existence of dynamic equilibria and the relation of these equilibria to Nash equilibria of associated $N$-player games are considered.
	In \citet{NeumannComputation} then the existence of stationary equilibria as well as several tools for the computation of these equilibria has been considered.
	
	We remark that stationary equilibria are of interest for several reasons:
	The computation of dynamic equilibria is impossible for an infinite time horizon, namely even the underlying individual control problem is intractable \citep{NeumannMyopic}, and for a finite time horizon it is equivalent to solving a forward-backward system of differential equations, which are notoriously intractable \citep{TrierPreprint}.
	However, we see (at least in an example) that stationary equilibria are under certain conditions limit objects of dynamic equilibria if the time horizon is large \citep{KolokoltsovBotnetCorruption}. 
	Moreover, we observe that stationary equilibria are limit points of a partially rational learning rule, the myopic adjustment process \citep{NeumannMyopic}.

	Now that results regarding existence and computation of stationary equilibria are available and that there is evidence that stationary equilibria are a sensible prediction of agents' behaviour in these games, a natural next step is to understand what happens to equilibria if the game is slightly perturbed. 
	More precisely, we are interested in essential equilibria, which are equilibria such that any perturbed game that is close to the original game has an equilibrium close to the considered equilibrium of the original game.
	This notion has been introduced by \citet{WenTsunEssential1962} in the context of normal form games with finite strategy spaces and has also been considered for static games with infinite action spaces \citep{YuInfiniteAction,CarbonellInfiniteAction,ScalzoInfiniteAction}, static population games \citep{CorreaPopuation} and Markov perfect equilibria \citep{DoraszelskiMPE2010}.
	
	This paper introduces essential equilibria in the setting of \citet{NeumannComputation} by providing a topological structure on the set of games and equilibria.
	We prove that the set of essential games, which are all those games where all equilibria are essential, is residual.
	The proof follows the classical line of argument, namely, we show that  the equilibrium correspondence is upper semicontinuous and that games with only essential equilibria are the points of continuity of this map.
	The classical theorem of \citet[Section 7]{FortSemiContinuity} then yields the desired result.
	Furthermore, we provide criteria to identify essential equilibria. 
	The first is a simple and classical consequence of the genericity statement, namely that unique equilibria are essential.
	Thereafter, we provide a deeper analysis of the results on equilibrium computation, which yields another criterion to obtain essentiality of equilibria.
	More precisely, we obtain that equilibria with a deterministic equilibrium strategy are essential when the equilibrium distribution is well-behaved (for example unique or an essential fixed point of an associated map).
	
	The remainder of the paper is structured as follows: In Section \ref{Section:Model} we introduce the considered model and define the notion of an essential equilibrium. Moreover, we review the relevant results regarding equilibrium computation. In Section \ref{Section:Generic} we prove that the set of all games with only essential stationary equilibria is residual and that games with a unique stationary equilibrium are essential. In Section \ref{Section:Further} we then introduce the second essentiality criterion. 
		
	\section{Stationary Equilibria of Mean Field Games with Finite State and Action Space}
	\label{Section:Model}
	
	In this section we introduce the considered mean field games model as in \citet{NeumannComputation, NeumannMyopic}. For details (in particular regarding well-definition and intuitions) we refer the reader to \citet{NeumannComputation}.  Moreover, we introduce in this section the notion of essential equilibria.
	
	Let $\mathcal{S}=\{1, \ldots, S\}$ ($S>1$) be the set of possible states of each player and let $\mathcal{A}=\{1, \ldots, A\}$ be the set of possible actions. 
	With $\mathcal{P}(\mathcal{S})$ we denote the probability simplex over $\mathcal{S}$ and with $\mathcal{P}(\mathcal{A})$ the probability simplex over $\mathcal{A}$. 
	A \textit{(mixed) strategy} is a measurable function $\pi: \mathcal{S} \times [0,\infty) \rightarrow \mathcal{P}(\mathcal{A})$, $(i,t) \mapsto (\pi_{ia}(t))_{a \in \mathcal{A}}$ with the interpretation that $\pi_{ia}(t)$ is the probability that at time $t$ and in state $i$ the player chooses action $a$. 
	A strategy $\pi=d:\mathcal{S} \times [0,\infty) \rightarrow \mathcal{P}(\mathcal{A})$ is \textit{deterministic} if for all $t \ge 0$ and for all $i \in \mathcal{S}$ that there is an $a \in \mathcal{A}$ such that $d_{ia}(t)=1$ and $d_{ia'}=0$ for all $a' \in \mathcal{A} \setminus \{a\}$. 
	Alternatively, we can represent a deterministic strategy equivalently by $d: \mathcal{S} \times [0,\infty) \rightarrow \mathcal{A}, (i,t) \mapsto d_i(t)$ with the interpretation that $d_i(t)=a$ states that at time $t$ in state $i$ action $a$ is chosen. 
	A \textit{stationary strategy} is a map $\pi: \mathcal{S} \times [0,\infty) \rightarrow \mathcal{P}(\mathcal{A})$ such that $\pi_{ia}(t) = \pi_{ia}$ for all $t \ge 0$. 
	With $\Pi$ we denote the set of all (mixed) strategies and with $D$ the set of all deterministic strategies. By $\Pi^s$ ($D^s$) we denote the set of all stationary (deterministic) strategies.

	The individual dynamics of each player given a Lipschitz continuous flow of population distributions $m: [0, \infty) \rightarrow \mathcal{P}(\mathcal{S})$ and a strategy $\pi: \mathcal{S} \times [0, \infty) \rightarrow \mathcal{P}(\mathcal{A})$ are given as a Markov process $X^\pi(m)$ with given initial distribution $x_0 \in \mathcal{P}(\mathcal{S})$ and infinitesimal generator given by the $Q(t)$-matrix $$\left( Q^\pi(m(t),t) \right)_{ij} = \sum_{a \in \mathcal{A}} Q_{ija}(m(t)) \pi_{ia}(t),$$  where for all $a \in \mathcal{A}$ and $m \in \mathcal{P}(\mathcal{S})$ the matrices $(Q_{\cdot \cdot a}(m))_{a \in \mathcal{A}}$ are conservative generators, that is $Q_{ija}(m) \ge 0$ for all $i, j \in \mathcal{S}$ with $i \neq j$ and $\sum_{j \in \mathcal{S}} Q_{ija}(m) =0$ for all $i \in \mathcal{S}$.

	The goal of each player is to maximize his expected discounted reward, which is given by 
	\begin{equation}
	\label{valueFunction}
	V_{x_0}(\pi, m) = \mathbb{E} \left[\int_0^\infty \left(\sum_{a \in \mathcal{A}} r_{X^\pi_t(m)a} (m(t)) \pi_{X^\pi_t(m) a}(t) \right) e^{-\beta t} \text{d}t\right],
	\end{equation} where $r: \mathcal{S} \times \mathcal{A} \times \mathcal{P}(\mathcal{S}) \rightarrow \mathbb{R}$ is a real-valued function and $\beta \in (0,1)$ is the discount factor. 
	That is, for a fixed flow of population distributions $m: [0,\infty) \mapsto \mathcal{P}(\mathcal{S})$ we face a Markov decision process with expected discounted reward criterion and time-inhomogeneous reward functions and transition rates. 
	
	We will work under the following assumption, which ensures that the model is well-defined and that dynamic as well as stationary equilibria exist \citep{NeumannComputation, DoncelPaper}:
	
	\begin{assumption}
		\label{assumption:continuous}
		For all $i, j \in \mathcal{S}$ and all $a \in \mathcal{A}$ the function $m \mapsto Q_{ija}(m)$ mapping from $\mathcal{P}(\mathcal{S})$ to $\mathbb{R}$ is Lipschitz-continuous in $m$ . 
		For all $i \in \mathcal{S}$ and all $a \in \mathcal{A}$ the function $m \mapsto r_{ia}(m)$ mapping from $\mathcal{P}(\mathcal{S})$ to $\mathbb{R}$ is continuous in $m$.
	\end{assumption}
	
	By $\mathcal{G}$ we denote the set of all games satisfying assumption \ref{assumption:continuous} and we denote a particular game by $(Q,r)$. 
	Moreover, we equip $\mathcal{G}$ with the following metric
	$$d((Q,r),(Q',r')) = \sup_{i,j \in \mathcal{S}, a \in \mathcal{A}, m \in \mathcal{P}(\mathcal{S})} |Q_{ija}(m)-Q_{ija}'(m)| + \sup_{i \in \mathcal{S}, a \in \mathcal{A}, m \in \mathcal{P}(\mathcal{S})} |r_{ia}(m)-r_{ia}'(m)|.$$
	
	\begin{remark}
		The space $(\mathcal{G},d)$ is a complete metric space since it is a closed subset of the space of all continuous functions $(Q,r)$.
	\end{remark}
	
	\begin{definition}
		\label{definition:Stationary}
		A \textit{stationary mean field equilibrium} is a pair $(m,\pi)$ consisting of a vector $m \in \mathcal{P}(\mathcal{S})$ and a stationary strategy $\pi \in \Pi^s$ such that
		\begin{itemize}
			\item for all $t \ge 0$ the marginal distribution of the process $X^\pi(m)$ at time $t$ is given by $m$ 
			\item for any initial distribution $x_0 \in \mathcal{P}(\mathcal{S})$ we have $V_{x_0}(\pi, m) \ge V_{x_0}(\pi', m)$ for all $\pi' \in \Pi$.
		\end{itemize}
	\end{definition}
	
	\begin{remark}
		Since for stationary strategies $\pi \in \Pi^s$ the matrix $Q_{ij}^\pi(m,t)$ does not depend on $t$ we write $Q^\pi_{ij}(m):= Q^\pi_{ij}(m,t)$. 
		Using this, we obtain that the first condition is equivalent to $$0 = \sum_{i \in \mathcal{S}}  m_i Q^\pi_{ij}(m) \quad \forall j \in \mathcal{S}.$$
	\end{remark}

	Any stationary equilibrium lies in the space $\mathcal{P}(\mathcal{S}) \times \mathcal{P}(\mathcal{A})^S$, which we equip with the maximum norm.
	Furthermore, we define the map $\text{SMFE}: \mathcal{G} \rightarrow 2^{\mathcal{P}(\mathcal{S}) \times \mathcal{P}(\mathcal{A})^S}$ as the map that maps every game $(Q,r)$ to the set of all stationary mean field equilibria of the game. Since by \citet[Theorem 4.1]{NeumannComputation} under Assumption \ref{assumption:continuous} at least one stationary equilibrium exists, this map is a well-defined set-valued map.
	
	In order to define the notion of essential equilibria we define the following notation for an arbitrary metric space $(X,d)$: For $A \subseteq X$ and $\epsilon>0$ we set
	$$N_\epsilon(A) = \{x \in X | \exists y \in A: d(x,y) < \epsilon\}$$ and on $2^X$ we define the Hausdorff metric by
	$$H(A,B) = \inf \{ \epsilon>0 : A \subseteq N_\epsilon(B) \wedge B \subseteq N_\epsilon(A)\}.$$
	
\begin{definition}
	Let $(m,\pi)$ be a stationary mean field equilibrium of the game $(Q,r)$. 
	We say that $(m, \pi)$ is \textit{essential} if for every $\epsilon >0$ there exists a $\delta >0$ such that for all games $(Q',r') \in N_\delta (Q,r)$ we have $(m, \pi) \in N_\epsilon( \text{SMFE}(Q',r'))$, which means that there is a stationary mean field equilibrium of $(Q', r')$ that lies in $N_\epsilon(m, \pi)$.
	We say that a game is essential if all stationary equilibria of the game are essential.
\end{definition}

	\section{Essential Games are Generic}
	\label{Section:Generic}
	
	This chapter presents the first main result of the paper, namely, that the set of all essential games is residual.
	The proof is based on the classical theorem by \citet{FortSemiContinuity} that the set of points of continuity of an upper semicontinuous map is residual. 
	Namely, we show that the equilibrium correspondence is upper semicontinuous and that a game is a point of continuity of the correspondence if and only if it is essential.
	
	\begin{theorem}
		\label{Generic:Theorem}
		The set of all games for which all equilibrium points are essential is residual in the set of all games $\mathcal{G}$. 
		Moreover, this set lies dense in $\mathcal{G}$. 
	\end{theorem}
	
	\begin{proof}
		\textit{Step 1: The map $\text{SMFE}(\cdot)$ is upper semicontinuous.}
		
		Let $(Q,r) \in \mathcal{G}$ and $\epsilon>0$. Assume that there is no $\delta>0$ such that any game $(Q',r') \in N_\delta(Q,r)$ satisfies $\text{SMFE}(Q',r')\subseteq N_\epsilon(\text{SMFE}(Q,r))$. Then we find sequences $(Q^n, r^n)_{n \in \mathbb{N}}$ and $(m^n,\pi^n)_{n \in \mathbb{N}}$ such that
		\begin{itemize}
			\item $d((Q^n,r^n),(Q,r))< \frac{1}{n}$
			\item $(m^n, \pi^n) \in \text{SMFE}(Q^n,r^n)$ for all $n \in \mathbb{N}$
			\item $(m^n, \pi^n) \notin N_\epsilon(\text{SMFE}(Q,r))$.
		\end{itemize}
	
		Since $\mathcal{P}(\mathcal{S}) \times \mathcal{P}(\mathcal{A})^S$ is compact we find a converging subsequence $(m^{n_k^1},\pi^{n_k^1})_{k \in \mathbb{N}}$ of $(m^n,\pi^n)_{n \in \mathbb{N}}$. Let $(m, \pi)$ be its limit.
		Let $A_1^k \times \ldots \times A_S^k \subseteq \mathcal{A}^S$ be such that $\pi_{ia}^{n_k^1}>0$ for all $i \in \mathcal{S}$, $a \in A_i^k$ and $\pi_{ia}^{n_k^1}=0$ for all $i \in \mathcal{S}$ and $a \notin A_i^k$. 
		Since $\mathcal{A}$ is finite, we find a set $A_1 \times \ldots \times A_S$ that occurs infinitely often.
		Let $(m^{n_k^2}, \pi^{n_k^2})_{k \in \mathbb{N}}$ be the subsequence of $(m^{n_k^1},\pi^{n_k^1})_{k \in \mathbb{N}}$ that runs through all indices $k$ such that $A_1^k \times \ldots \times A_S^k = A_1 \times \ldots \times A_S$.
		
		By Theorem 3.2 in \citet{NeumannComputation} a stationary strategy $\pi \in \Pi^s$ is optimal if and only if it is a convex combination of optimal deterministic stationary strategies.		
		Since $(m^{n_k^2}, \pi^{n_k^2})$ is for each $m \in \mathbb{N}$ a stationary mean field equilibrium, this implies that for all strategies $d \in D^s$ such that $d(i) \in A_i$ for all $i \in \mathcal{S}$ we have $V(d,m^{n_k^2}) = V^\ast(m^{n_k^2})$, where $V^\ast(m)$ is the value function of the individual control problem.
		By \citet[Section 3]{NeumannComputation} the functions $V(d, \cdot)$ and $V^\ast(\cdot)$ are continuous. Thus, $V(d,m) = V^\ast(m)$ for all strategies $d \in D^s$ such that $d(i) \in A_i$ for all $i \in \mathcal{S}$.
		Since $\pi^{n_k^2}_{ia}=0$ for all $i \in \mathcal{S}, a \notin A_i, k \in \mathbb{N}$ we have $\pi_{ia}=0$ for all $i \in \mathcal{S}$ and $a \notin A_i$.
		Again using the fact that a stationary strategy is optimal if and only it is a convex combination of optimal deterministic stationary strategies, we obtain that $\pi$ is indeed optimal for $m$.
		
		Furthermore, by uniform convergence, we have that $$\sum_{i \in \mathcal{S}} \sum_{a \in \mathcal{A}} Q_{ija}(m) m_i \pi_{ia} \leftarrow \sum_{i \in \mathcal{S}} \sum_{a \in \mathcal{A}} Q_{ij}^{n_l} (m^{n_l}) m_i^{n_l} \pi_{ia}^{n_l} = 0.$$
		Thus, $m$ is a stationary point given $Q^\pi(\cdot)$, which in total yields that $(m,\pi) \in \text{SMFE}(Q,r)$.		
		However, $(m^{n_k^2}, \pi^{n_k^2}) \notin N_\epsilon( \text{SMFE}(Q,r))$ for all $n \in \mathbb{N}$ implies $(m, \pi) \notin N_\epsilon( \text{SMFE}(Q,r))$, a contradiction. Therefore, we conclude that $\text{SMFE}(\cdot)$ is indeed upper semicontinuous.
		
		\textit{Step 2: A game $(Q,r)$ is essential if and only if $(Q,r)$ is a point of continuity of $\text{SMFE}$.}
		
		Assume first that the game $(Q,r)$ is essential. Then, by definition, there is for each $(m,\pi) \in \text{SMFE}(Q,r)$ an $\delta_{(m,\pi)}>0$ such that all games $(Q',r') \in N_{\delta_{(m,\pi)}}(\text{SMFE}(Q,r))$ have a stationary mean field equilibrium in the $\frac{\epsilon}{2}$-neighbourhood of $(m, \pi)$.
		
		Let us first note that $\text{SMFE}(Q,r)$ is compact: 
		Indeed, let $(m^n,\pi^n)_{n \in \mathbb{N}}$ be a sequence in $\text{SMFE}(Q,r)$. 
		Since $\mathcal{P}(\mathcal{S}) \times \mathcal{P}(\mathcal{A})^S$ is compact, we find a converging subsequence $(m^{n_k},\pi^{n_k})_{k \in \mathbb{N}}$ with limit $(m,\pi) \in \mathcal{P}(\mathcal{S}) \times \mathcal{P}(\mathcal{A})^S$.
		We note that $\pi_{ia}>0$ if and only if $\pi_{ia}^{n_k}>0$ for infinitely many $k \in \mathbb{N}$. Thus, by the same argument as above $\pi$ is optimal for $m$. By uniform convergence we then obtain that $(m, \pi) \in \text{SMFE}(Q,r)$.
		
		By compactness of $\text{SMFE}(Q,r)$, there exists a finite set $\{(m^1, \pi^1), \ldots, (m^n, \pi^n)\} \subseteq \text{SMFE}(Q,r)$ such that each point in $\text{SMFE}(Q,r)$ lies within the $\frac{\epsilon}{2}$-neighbourhood of some point $m_k$, $k \in \{1, \ldots, n\}$. Choose $\delta = \min_{k \in \{1, \ldots, n\}} \delta_{(m^k,\pi^k)}$.
		
		Now let $(Q',r') \in N_\delta(Q,r)$ and let $(m,\pi) \in \text{SMFE}(Q,r)$. 
		Then by construction there is a point $(m^k,\pi^k)$ at most $\frac{\epsilon}{2}$ away from $(m,\pi)$. 
		By choice of $\delta$, we find equilibrium given $(Q',r')$ at most $\frac{\epsilon}{2}$ away from $(m^k,\pi^k)$.
		This yields that there is an equilibrium of the game $(Q',r')$ at most $\epsilon$ away from $(m,\pi)$. This proves, together with the first part, that $\text{SMFE}(\cdot)$ is continuous at $(Q,r)$.
		
		Let us now assume that $\text{SMFE}(\cdot)$ is continuous at $(Q,r)$. Let $(m, \pi) \in \text{SMFE}(Q,r)$ and $\epsilon>0$.
		By continuity of $\text{SMFE}(\cdot)$ at $(Q,r)$, we find an $\delta>0$ such that $d((Q,r),(Q',r'))< \delta$ implies that $H(\text{SMFE}(Q,r),\text{SMFE}(Q',r'))< \epsilon$. 
		In particular, this yields that all $(Q',r')\in N_\delta(Q,r)$ satisfy that $\text{SMFE}(Q,r) \subseteq N_\epsilon(\text{SMFE}(Q',r'))$, which implies that for any equilibrium $(m, \pi)$ of $(Q,r)$ there is $\epsilon$ an equilibrium of $(Q',r')$ in the $\epsilon$-neighbourhood of $(m, \pi)$.
		
		\textit{Step 3: Conclusion}
		
		The classical theorem of \citet[Section 7]{FortEssential1950} states that the points of continuity of any upper semi-continuous function $F(\cdot)$ ranging from a topological space to the power set of a separable metric space equipped with metric $H(\cdot,\cdot)$ is a $G_\delta$-residual set in the topological space.
		This theorem yields that the set of all essential games is a $G_\delta$-residual set. 
		Since $\mathcal{G}$ is a closed subset of a complete metric space, we obtain by Baire's Theorem \citep[p.200]{KelleyTopology1955} that this set is dense in $\mathcal{G}$.
	\end{proof}
	
	The following sufficient condition for essentiality is a classical and immediate consequence of the genericity result:
	
	\begin{corollary}
		\label{Uniqueness:Theorem}
		Let $(Q,r) \in \mathcal{G}$ and let $(m, \pi)$ be the unique equilibrium of the game $(Q,r)$. Then $(m, \pi)$ is essential.
	\end{corollary}

	\begin{proof}
		Let $\epsilon>0$. 
		Since $\text{SMFE}(\cdot)$ is upper semicontinuous there is an $\delta>0$ such that for any $(Q',r') \subseteq N_\delta(Q,r)$ we have $\text{SMFE}(Q',r') \subseteq N_\epsilon(\text{SMFE}(Q,r))$.
		The set $\text{SMFE}(Q',r')$ is non-empty because we find a stationary equilibrium for any game in $\mathcal{G}$. 
		This and the fact that $\text{SMFE}(Q,r)$ is a singleton yields that $\text{SMFE}(Q,r) \subseteq N_\epsilon(\text{SMFE}(Q',r'))$.
		This in turn implies that $\text{SMFE}(\cdot)$ is continuous at $(Q,r)$, which means that $(m,\pi)$ is essential.
	\end{proof}
	
	\section{A Second Characterization Result for Essential Stationary Equilibria}
	\label{Section:Further}
	
	In this section we will provide another characterization result for essential equilibria.
	It will rely on the results regarding equilibrium computation derived in \citet{NeumannComputation}.
	More precisely, for an equilibrium $(m, \pi)$ to be essential it is necessary that under a small perturbations there is still an equilibrium close to $(m,\pi)$.
	Since it holds that a stationary strategy is optimal if and only if it is a convex combination of optimal deterministic stationary strategies, a sufficient condition for an essential equilibrium is that all deterministic strategies that had positive weight in the equilibrium strategy still have to be optimal given the perturbation.
	However, we cannot expect that this holds for the equilibrium distribution $m$ and also in a neighbourhood this is only clear if at the equilibrium distribution at most two strategies are optimal.
	Moreover, in this neighbourhood there has to be a stationary strategy that is optimal for the perturbed game and has a stationary point $m'$ (i.e. $0=m^T(Q')^{\pi'}(m')$).
	However, the nonlinear equations $0=m^T Q^{\pi}(m)$ and $0=m^T(Q')^{\pi'}(m')$ do not relate for different strategies $\pi$ and $\pi'$.
	Thus, it is not possible to draw any conclusions for mixed strategy equilibria.
	Nonetheless, we obtain a criterion for deterministic strategy equilibria.
	
	To prove this criterion we will first analyse the two problems (the individual agent's control problem and the fixed point problem) individually and then combine the results.
	First, we prove that a deterministic strategy that is the unique optimal strategy for a certain point will be optimal in a neighbourhood of this point for all slightly perturbed games.
	Thereafter, we will analyse the fixed point problem and obtain that there are no full characterization results of essential stationary points, but similar results to \citet{FortEssential1950}.
	
	Let us start with a preliminary lemma to prove that deterministic stationary strategies that are the unique optimal strategy remain optimal under small perturbations.

	\begin{lemma}
		\label{OptSets:Preliminary}
		Let $(Q,r) \in \mathcal{G}$ be a game and let $\gamma>0$. Then there exists a $\delta>0$ such that for all $(Q',r') \in N_\delta(Q,r)$ and all deterministic stationary strategies $d \in D^s$ the distance of the expected discounted reward of the associated Markov decision process given strategy $d$ at any point $m \in \mathcal{P}(\mathcal{S})$ in the game $(Q,r)$ and in the perturbed game $(Q',r')$ is at most $\gamma$.
	\end{lemma} 
	
	\begin{proof}
		Without loss of generality we assume that $r_{ia}(m)>0$ for all $i \in \mathcal{S}$, $a \in \mathcal{A}$ and $m \in \mathcal{P}(\mathcal{S})$, else consider the game with rewards $r_{ia}(m) + (- \inf_{i \in \mathcal{S}, a \in \mathcal{A},m \in \mathcal{P}(\mathcal{S})} r_{ia}(m) +1 )$, which is equivalent.
		
		By \citet[Theorem 2.4]{KakumanuValueFunction1971} the expected discounted reward $V^d(m)$ of strategy $d \in D^s$ is given as the unique solution of $\left( \beta I-Q^d(m)\right) V^d(m) =r^d(m)$, where $r^d(m) = \left( r_{id(i)}(m)\right)_{i \in \mathcal{S}}$ and $Q^d(m) = \left( Q_{ijd(i)}(m)\right)_{i,j\in \mathcal{S}}$. 
		By \citet[Corollary C.4]{PutermanMDP1994} we moreover obtain that $\left( \beta I-Q^d(m)\right)$ is invertible.
		This allows us to apply the perturbation theorem \citep[Theorem 2.43]{WendlandNumericalLA2018}, which yields that for any game $(Q',r')$ such that $||Q^d(m)-(Q')^d(m)|| < \frac{1}{||Q(m)^{-1}||}$ it holds that
		\begin{align*}
			& ||(V')^d(m)-V^d(m)|| \\
			&\le ||V^d(m)|| \kappa (\beta I - Q^d(m)) \left( 1 - \kappa (\beta I -Q^d(m))
			\frac{||(Q')^d(m)-Q^d(m)||}{||\beta I - Q ^d(m)||} \right) ^{-1}\\
			&\quad \cdot \left( \frac{||(Q')^d(m)-Q^d(m)||}{||\beta I - Q ^d(m)||} + \frac{||(r')^d(m)-r^d(m)||}{||r^d(m)||} \right),
		\end{align*} where $\kappa(A):= ||A|| \cdot ||A^{-1}||$ is the conditioning number of the matrix $A$.
		
		Define
		\begin{align*}
			L_1 &:= \inf_{m \in \mathcal{P}(\mathcal{S}), d \in D^s} ||\beta I - Q^d(m)|| >0 \\
			L_2&:= \left( \sup_{m \in \mathcal{P}(\mathcal{S}), d \in D^s} || \beta I - Q^d(m)|| \right) \left( \sup_{m \in \mathcal{P}(\mathcal{S}), d \in D^s}  || \left(\beta I - Q^d(m) \right)^{-1} || \right) < \infty \\
			L_3 &:= \inf_{m \in \mathcal{P}(\mathcal{S}), d \in D^s} ||r^d(m)|| >0 \\
			L_4 &:= \sup_{m \in \mathcal{P}(\mathcal{S}), d \in D^s} ||V^d(m)|| >0
		\end{align*} and choose $\delta$ such that
		$$0 < \delta < \min \left\{ \min_{d \in D^s} \frac{1}{2 S ||(\beta I - Q^d(m))^{-1}||}, \frac{\gamma L_1 L_3}{2 L_2 L_4 (SL_3+L_1)} \right\}.$$
		Then 
		\begin{align*}
		& \left( 1 - \kappa(\beta I -Q^d(m)) \frac{||Q^d(m)-(Q')^d(m)||}{||\beta I -Q^d(m)||} \right)^{-1} \\
		&= \left( 1 - ||\beta I -Q^d(m))^{-1}|| \cdot ||Q^d(m) - (Q')^d(m)|| \right)^{-1} \\
		&\le \left( 1- || (\beta I -Q^d(m))^{-1}|| \cdot \frac{1}{2 || (\beta I - Q^d(m))^{-1}||} \right)^{-1} \\
		&= \left(1- \frac{1}{2}\right)^{-1} = 2. 
		\end{align*}
		Using this we obtain
		\begin{align*}
		& ||(V')^d(m)-V^d(m)|| \\
		&\le L_4 \cdot L_2 \cdot 2 \cdot  \left( \frac{||(Q')^d(m)-Q^d(m)||}{L_1} + \frac{||(r')^d(m)-r^d(m)||}{L_3} \right) \\
		&\le L_4 \cdot L_2 \cdot 2 \cdot  \left( \frac{S d((Q,r),(Q',r'))}{L_1} + \frac{ d((Q,r),(Q',r'))}{L_3} \right) \\
		& = d((Q,r),(Q',r')) \cdot \frac{2 L_2 L_4 (S L_3 +  L_1)}{L_1 \cdot L_3} \le \gamma.
		\end{align*}		
	\end{proof}

	With these preparations we can prove the announced result regarding optimality of deterministic strategies under small perturbations. For this let $\mathcal{D}(m)$ denote the set of all optimal deterministic stationary strategies, i.e. those strategies $d \in D^s$ that maximize $V(\cdot,m)$. We remark that they can be explicitly characterized and that the set $\mathcal{D}(m)$ is non-empty for all $m \in \mathcal{P}(\mathcal{S})$ \citep[Theorem 3.2]{NeumannComputation}.

	\begin{lemma}
		\label{comparative:OptimalitySetsSmallPerturbationPositive}
		Let $(Q,r) \in \mathcal{G}$ be a game.
		Assume that $d \in D^s$ is the unique optimal deterministic strategy for $m$, that is $\mathcal{D}(m) = \{d\}$. 
		Then there is an $\epsilon>0$ such that for all $m' \in \overline{N_\epsilon(m)}$ we have $\mathcal{D}(m')=\{d\}$. 
		Furthermore, for any such $\epsilon$ there is a $\delta>0$ such that for all $(Q',r') \in N_\delta(Q,r)$ and all $m' \in N_\epsilon(m)$ we have $\mathcal{D}'(m')=\{d\}$. 
	\end{lemma}

	\begin{proof}
		Since $\mathcal{D}(m) = \{d\}$ we have that $V^d(m)>V^{\hat{d}}(m)$ pointwise for all $\hat{d} \in D^s \setminus \{d\}$. 
		By continuity of $V^d(\cdot)$ and $V^{\hat{d}}(\cdot)$ \citep[Section 3]{NeumannComputation}  we furthermore find an $\epsilon>0$ such that 
		$V^d(m') > V^{\hat{d}} (m')$ pointwise for all $m' \in \overline{N_\epsilon(m)}$ and all $\hat{d} \in D^s \setminus \{d\}$. 
		Since $\overline{N_\epsilon(m)}$ is compact, we have that
		$$\inf_{m' \in \overline{N_\epsilon(m)}} || V^d(m')-V^{\hat{d}}(m')|| = \gamma >0.$$ 
		By Lemma \ref{OptSets:Preliminary} there is a $\delta>0$ such that for all games $(Q',r') \in N_\delta(Q,r)$, all strategies $\hat{d} \in D^s$ and all $m' \in \mathcal{P}(\mathcal{S})$ we have $||V^{\hat{d}}(m')-(V')^{\hat{d}} (m')|| < \frac{\gamma}{3}.$ 
		With this it holds pointwise for $\hat{d} \in D^s \setminus \{d\}$ and $m' \in N_\epsilon(m)$ that
		$$(V')^d(m') - (V')^{\hat{d}}(m') > V^d(m') - \frac{\gamma}{3} 1  - V^{\hat{d}}(m') - \frac{\gamma}{3} 1 = V^d(m')-V^{\hat{d}}(m') - \frac{2 \gamma}{3}> 0.$$ 
		Thus, $d$ is the only optimal strategy in the perturbed game for $m'$.
	\end{proof}
	
	The analysis of the fixed point problem is more complex. 
	The question whether the equilibrium distribution is stable with respect to slight perturbations of the transition rates cannot be answered completely.
	Instead, it is closely linked to the question whether the fixed point of a certain map is essential.
	
	Let us introduce some notation:
	Let $\text{TR}$ be the set of all transition rate matrix function that can occur in our game, that is let it be the set of all functions $Q: \mathcal{P}(\mathcal{S}) \rightarrow \mathbb{R}^{S \times S}$ such that $Q(m)$ is a generator for all $m \in \mathcal{P}(\mathcal{S})$ and $Q_{ija}(m)$ is Lipschitz continuous in $m$. 
	Let us call the solutions of $0=m^TQ(m)$ in $\mathcal{P}(\mathcal{S})$ \textit{stationary points of $Q$}.

	\begin{definition}
		Let $Q \in \text{TR}$ and let $m$ be a stationary point given $Q$. We say that this point is an \textit{essential stationary point} if for all $\epsilon>0$ there is an $\delta>0$ such that all $Q' \in N_\delta(Q)$ have a stationary point in $N_\epsilon(m)$.
	\end{definition}
	
	Furthermore, let $x(\cdot)$ be the map that maps $m \in \mathcal{P}(\mathcal{S})$ to all solutions of $0=x^TQ(m)$. 
	We remark that $m$ is a stationary point of $Q$ if and only if $m$ is a fixed point of $x(\cdot)$.
	Then we obtain the following relation to essential fixed points:
	
	\begin{lemma}
		Let $Q \in \text{TR}$ and assume that $Q(m)$ is irreducible for all $m \in \mathcal{P}(\mathcal{S})$. 
		Furthermore, assume that $m$ is an essential fixed point of the map $x(\cdot)$. 
		Then $m$ is an essential stationary point. 
	\end{lemma}

	\begin{proof}
		Since $Q(\cdot)$ is irreducible, $x(\cdot)$ is a function, not a set-valued map.
		Moreover, it was shown in \citet{NeumannComputation} that $x(m)$ is the unique solution of $\tilde{Q}(m)x(m) = (0, \ldots, 0,1)^T$  for a sensibly defined and invertible matrix $\tilde{Q}(m)$.
		It suffices to prove that for any $\epsilon>0$ there is a $\delta>0$ such that if $Q' \in N_\delta(Q)$, then $\rho(x,x')< \epsilon$.
		Since $\tilde{Q}(m)$ is invertible, the conclusion follows as in Theorem \ref{OptSets:Preliminary} by using the perturbation theorem.
	\end{proof}
	
	We remark that this allows us to apply the known characterizations for essential fixed points (see \citet{FortEssential1950}) to our setting.
	We obtain that if $Q(m)$ is irreducible for all $m \in \mathcal{P}(\mathcal{S})$ and there is a unique stationary point, then this stationary point is essential.
	Moreover, we obtain that if $Q(m)$ is irreducible for all $m \in \mathcal{P}(\mathcal{S})$ and the set of all stationary points is totally disconnected, then there is at least one essential stationary point.

	Moreover, we can readopt the proofs of Theorem \ref{Generic:Theorem} and \ref{Uniqueness:Theorem} to prove the following statements:

	\begin{lemma}
		The set of all $Q \in \text{TR}$ for which all stationary points are essential is residual in the set $\text{TR}$. Moreover, this set lies dense in $\text{TR}$.
	\end{lemma}

	\begin{lemma}
		Let $Q \in \text{TR}$ and let $m$ be the unique stationary point of $Q$. Then $m$ is essential. 
	\end{lemma}
	
	With these preparations we can state and prove the second main result:
	
	\begin{theorem}
		\label{Chracterization}
		Let $(Q,r) \in \mathcal{G}$ and let $(m, d)$ be an equilibrium of $(Q,r)$ such that $d \in D^s$ is the unique equilibrium strategy given $m$ and such that $m$ is an essential stationary point given $Q(\cdot)$.
		Then $(m,d)$ is essential.
	\end{theorem}

	\begin{proof}
		Since $d$ is the unique optimal strategy given $m$ there is an $\bar{\epsilon}$ such that for all $m' \in \overline{N_{\bar{\epsilon}}(m)}$ we have $\mathcal{D}(m')=\{d\}$. 
		Let $\epsilon \in (0,\bar{\epsilon})$.
		By Lemma  \ref{comparative:OptimalitySetsSmallPerturbationPositive}, we find a $\delta_1>0$ such that for any game $(Q',r') \in N_{\delta_1}(m)$ and all $m' \in N_\epsilon(m)$ we have $\mathcal{D}'(m')=\{d\}$.
		Since $m$ is an essential stationary point of the dynamics, there is a $\delta_2>0$ such that for any $(Q',r') \in N_{\delta_2}(Q,r)$ there is a stationary point of the dynamics given $(Q')^d(\cdot)$ in $N_\epsilon(m)$.		
		Choosing $\delta= \min \{\delta_1, \delta_2\}$ yields the desired result.
	\end{proof}
	
	We highlight that Theorem \ref{Chracterization} indeed yields different conclusions than Theorem \ref{Uniqueness:Theorem}. 
	The first result allowed to characterize any equilibrium (also a mixed strategy equilibrium) as an essential equilibrium whenever it is the unique equilibrium of the game.
	This second result now yields that any deterministic equilibrium is essential whenever the deterministic equilibrium strategy is the unique optimal strategy for the equilibrium distribution and the equilibrium distribution is an essential stationary point of $Q^d$.
	In particular, Theorem \ref{Chracterization} yields that several deterministic equilibria $(m^{d^i}, d^i)_{i=1}^n$ each with a different equilibrium strategy are essential whenever the points $m^{d^i}$ are the unique stationary point of $Q^{d^i}$.
	
	We conclude by describing the use of the characterization results in several examples:
	In the consumer choice model discussed in \citet{NeumannComputation} Theorem \ref{Chracterization} allows to characterize all deterministic stationary equilibria as essential whenever we are outside the knife-edge cases $k_1 \in \{ \frac{\epsilon}{b+\epsilon}, \frac{1}{2}\}$ or $k_2 \in \{ \frac{1}{2}, \frac{b}{b+\epsilon} \}$. 
	In the consumer choice model with congestion effects discussed in \citet{NeumannDissertation} we obtain using Theorem \ref{Uniqueness:Theorem} that if there is a unique equilibrium then it is essential.
	Also if we consider the botnet defence model of \citet{KolokoltsovCorruption2017} in the discounted cost formulation we obtain that all deterministic equilibria such that for the equilibrium distribution only one strategy is optimal are essential.

%	The bibliography will go here
	\bibliographystyle{plainnat}
	\bibliography{literatureStability}

\end{document}